\documentclass{article}
\usepackage[utf8]{inputenc}

\title{Subordination Algebras as Semantic Environment of Input/Output Logic}
\usepackage{graphicx}
\usepackage{amsmath}
\usepackage{amssymb}
\usepackage{amsfonts}
\usepackage{amsthm}
\usepackage{url}
\usepackage[all]{xy}
\usepackage{makeidx}
\usepackage{latexsym}
\usepackage{empheq} 
\usepackage{verbatim}
\usepackage{stmaryrd}
\usepackage{enumitem}
\usepackage{bussproofs}
\usepackage{boxedminipage}
\usepackage{hyperref}
\usepackage{ccaption}
\usepackage{cancel}
\usepackage{multicol}
\usepackage{xcolor}
\usepackage{proof}
\usepackage{longtable}
\usepackage{marginnote}
\newcommand{\bba}{\mathbb{A}}
\newcommand{\bbas}{\mathbb{A}^\delta}
\newcommand{\nomj}{\mathbf{j}}
\newcommand{\nomk}{\mathbf{k}}\newcommand{\nomi}{\mathbf{i}}
\newcommand{\cnomm}{\mathbf{m}}
\newcommand{\cnomn}{\mathbf{n}}
\newcommand{\cnomo}{\mathbf{o}}
\newcommand{\jty}{J^{\infty}}
\newcommand{\mty}{M^{\infty}}
\newcommand{\adnote}[1]{\textcolor{blue}{ADD:#1}}

\newcommand{\mpnote}[1]{\textcolor{red}{MP:#1}}

\newcommand{\marginkmnote}[1]{\marginnote{\textcolor{green}{KM:#1}}}
\newcommand{\afnote}[1]{\textcolor{orange}{AF:#1}}

\newtheorem{theorem}{Theorem}[section]
\newtheorem{lemma}[theorem]{Lemma}

\newtheorem{definition}[theorem]{Definition}
\newtheorem{corollary}[theorem]{Corollary}
\newtheorem{proposition}[theorem]{Proposition}
\newtheorem{remark}[theorem]{Remark}

\author{Andrea De Domenico  \and
Ali Farjami  \and
Krishna Manoorkar  \and
Alessandra Palmigiano  \and
Mattia Panettiere  \and
Xiaolong Wang }

\date{}

\begin{document}

\maketitle
\begin{abstract}
  We establish a novel connection between two research areas in non-classical logics which have been  developed independently of each other so far: on the one hand,  {\em input/output logic},  introduced within a research program developing  logical formalizations of normative reasoning in philosophical logic and AI; on the other hand, {\em subordination algebras}, investigated in the context of a research program integrating topological, algebraic, and duality-theoretic techniques in the study of the semantics of modal logic. Specifically, we propose that the basic framework of input/output logic, as well as its extensions, can be given formal semantics  on (slight generalizations of) subordination algebras. The existence of this interpretation brings benefits to both research areas: on the one hand, this connection allows for a novel conceptual understanding of  subordination algebras as  mathematical  models of the properties and  behaviour of  norms; on the other hand, thanks to the well developed connection between subordination algebras and modal logic, the output operators in input/output logic can be given a new formal representation as modal operators, whose properties can be explicitly axiomatised in a suitable language, and be systematically  studied by means of mathematically established
  and powerful tools.  
  %
  \end{abstract}

\section{Introduction} 
Input/output logic \cite{Makinson00} has been introduced  as a formal framework for modelling the interaction between logical inferences and other agency-related notions  such as  conditional obligations, goals, ideals, preferences, actions, and beliefs. 
This framework has been applied mainly in the context of the formalization of normative systems in philosophical logic and AI. Although, initially, this framework was intended ``not [for] studying some kind of non-classical logic, but [as] a way of using the classical one'', its generality and versatility makes it very suitable to support a range of  enhancements in its expressiveness, such as those brought about by the addition of modal operators. Moreover, recently, there has been an interest in studying the interaction between the agency-related notions mentioned above with various forms of {\em nonclassical} reasoning \cite{parent2014intuitionistic,stolpe2015concept}. This interest has  contextually motivated the introduction of algebraic and proof-theoretic methods in the study of input/output logic   \cite{sun2018proof}.

In this paper, we contribute to the latter  research direction in the mathematical background of input/output logic by introducing an algebraic semantics for it, based on (generalizations of) {\em subordination algebras} \cite{bezhanishvili2017irreducible}. These can be defined as tuples $(A, \prec)$ such that $A$ is a Boolean algebra and $\prec$ is a binary relation on $A$  such that the direct (resp.~inverse) image of each element $a\in A$ is a filter (resp.~an ideal) of $A$. Subordination algebras  are equivalent presentations of pre-contact algebras \cite{dimov2005topological} and quasi-modal algebras \cite{celani2001quasi,celani2016precontact}. Since their introduction, subordination algebras have been systematically connected with various modal algebras (i.e.~Boolean algebras expanded with semantic modal operators). This has made it possible to endow  various modal languages with algebraic semantics  based on subordination algebras, and use these languages to axiomatize the properties of these subordination algebras. In particular, Sahlqvist-type canonicity for modal and tense formulas on subordination algebras has been studied in \cite{de2020subordination} using topological techniques;  in \cite{de2021slanted}, using algebraic techniques, the canonicity result  of \cite{de2020subordination} was strengthened and captured within the more general notion of canonicity in the context of {\em slanted algebras}, which was established using the tools of {\em unified correspondence theory} \cite{conradie2014unified,conradie2019algorithmic,conradie2020constructive}. Slanted algebras are based on general lattices, and encompass variations and generalizations of subordination algebras such as those very recently introduced by Celani in \cite{celani2020subordination}, which are based on distributive lattices, and for which Celani develops  duality-theoretic and correspondence-theoretic results.

\paragraph{Structure of the paper.} In Section \ref{sec:prelim}, we collect basic definitions and facts about the abstract logical framework in which we are going to develop our results, input/output logics as embedded in this framework, the general environment of proto-subordination algebras and their properties, canonical extensions and slanted algebras. In Section \ref{sec: proto and slanted}, we associate slanted algebras to proto-subordination algebras with certain properties, and characterize their further properties  in terms of the validity of modal inequalities on their associated slanted algebras. In Section \ref{sec: applications}, we use the characterizations presented in the previous section to provide an axiomatic modal characterization of the output operators of input/output logic (cf.~Proposition \ref{prop:output}), and to obtain Celani's dual characterization results for subordination lattices as consequences of standard modal correspondence (cf.~Proposition \ref{prop: celani}). We conclude in Section \ref{sec:conclusions}.
\section{Preliminaries}
\label{sec:prelim}
\subsection{Selfextensional logics}\label{ssec:selfextensional}
In what follows, we align to the literature in abstract algebraic logic \cite{font2003survey}, and  understand a {\em logic} to be a tuple 
$\mathcal{L}= (\mathrm{Fm}, \vdash)$, such that $\mathrm{Fm}$ is the term algebra (in a given algebraic signature) over a set $\mathsf{Prop}$ of atomic propositions, and $\vdash$ is a {\em consequence relation} on $\mathrm{Fm}$, i.e.~$\vdash$ is a relation between sets of formulas and formulas  such that, for all $\Gamma, \Delta\subseteq \mathrm{Fm}$ and all $\varphi\in \mathrm{Fm}$, (a) if $\varphi\in \Gamma$ then $\Gamma\vdash \varphi$; (b) if  $\Gamma\vdash \varphi$ and $\Gamma \subseteq \Delta$, then $\Delta\vdash \varphi$; (c) if $\Delta\vdash \varphi$ and $\Gamma\vdash \psi$ for every $\psi\in \Delta$, then $\Gamma\vdash \varphi$. 
Clearly,  any such $\vdash$ induces a preorder on $\mathrm{Fm}$, which we still denote $\vdash$, by restricting to singletons.    A logic $\mathcal{L}$ is {\em selfextensional} (cf.~\cite{jansana2006referential}) if the relation ${\equiv} \subseteq\, \mathrm{Fm}\times \mathrm{Fm}$, defined by $\varphi\equiv \psi$ iff $\varphi\vdash \psi$ and $\psi\vdash \varphi$, is a congruence of $\mathrm{Fm}$. In this case, the {\em Lindenbaum-Tarski algebra} of $\mathcal{L}$ is the partially ordered algebra $Fm = (\mathrm{Fm}/{\equiv}, \vdash)$ where, abusing notation, $\vdash$ also denotes the partial order on $\mathrm{Fm}/{\equiv}$, defined as $[\varphi]_{\equiv}\vdash [\psi]_{\equiv}$ iff $\varphi\vdash \psi$.
In what follows, we will also assume that each element in the class $\mathsf{Alg}(\mathcal{L})$ of algebras canonically associated with $\mathcal{L}$ is partially ordered, and  that, if $\varphi$ and $\psi$ are formulas, then $\varphi\vdash \psi$ iff $h(\varphi)\leq h(\psi)$ for every $A\in \mathsf{Alg}(\mathcal{L})$ and every homomorphism $h: \mathrm{Fm}\to A$.

For any $\Gamma\subseteq \mathrm{Fm}$, let $Cn(\Gamma): = \{ \psi \mid \Gamma \vdash \psi \}$.\footnote{In what follows, we write e.g.~$Cn(\varphi)$ for $Cn(\{\varphi\})$.} 
The {\em conjunction property} holds for $\mathcal{L}$ if a term $t(x, y): = x\wedge y$ exists such that $Cn(\varphi\wedge \psi) = Cn(\{\varphi, \psi\})$ for all $\varphi, \psi\in \mathrm{Fm}$. The {\em disjunction property} holds for $\mathcal{L}$ if a term $t(x, y): = x\vee y$ exists such that $Cn(\varphi\vee \psi) = Cn(\varphi)\cap Cn( \psi)$ for all $\varphi, \psi\in \mathrm{Fm}$.

Although the original framework of input/output logic takes $\mathcal{L}$ to be classical propositional logic, in the next subsection we  present it in the more general framework of selfextensional logics just described.  
\subsection{Input/output logic}
The general theory of input/output logic aims at modelling relations generalizing inference, where inputs need not be included among outputs, and outputs need not to be reusable as inputs \cite{Makinson00}.  
\begin{definition}
Let $\mathcal{L}= (\mathrm{Fm}, \vdash)$ be a logic in the sense specified above. 
A  {\em normative system} is a relation $N \subseteq \mathrm{Fm}\times \mathrm{Fm}$, the elements  $(\alpha,\varphi)$ of which are called {\em conditional norms} (or obligations). An {\em input/output logic} is a tuple $\mathbb{L} = (\mathcal{L}, N)$ s.t.~$\mathcal{L}= (\mathrm{Fm}, \vdash)$ is a (selfextensional) logic, and $N$ is a normative system on $\mathrm{Fm}$.
\end{definition}
The reading of each  norm $(\alpha,\varphi)\in N$ is ``given $\alpha$, it is obligatory that $\varphi$''. The formula $\alpha$ is  the {\em body} of the norm, and represents some situation or condition, while $\varphi$ is  the {\em head} and represents what is obligatory or desirable in that situation.
For any  $\Gamma\subseteq \mathrm{Fm}$, let $N(\Gamma) := \{ \psi \mid \exists \alpha(\alpha \in \Gamma \ \&\ (\alpha,\psi)\in N )\}$. 

\begin{definition}[Output operations] For any input/output logic $\mathbb{L} = (\mathcal{L}, N)$, and each $1\leq i\leq 4$,
the output operation $out_i^N$ is defined as follows: for any $\Gamma\subseteq \mathrm{Fm}$,
\[out_i^N(\Gamma): = N_i(\Gamma) = \{ \psi\in \mathrm{Fm} \mid \exists \alpha(\alpha \in \Gamma \ \&\ (\alpha,\psi)\in N_i )\} \] 
 where $N_i\subseteq \mathrm{Fm}\times \mathrm{Fm}$ is the  {\em closure}  of $N$ under (i.e.~the smallest extension of $N$ satisfying)  the  inference rules below, as  specified in the table.

\vspace{1mm}
\begin{center}
\begin{tabular}{lll}
\infer[(\top)]{(\top,\top)}{} &
\infer[\mathrm{(SI)}]{(\beta,\varphi)}{(\alpha,\varphi) &  \beta \vdash \alpha} &
\infer[\mathrm{(WO)}]{(\alpha,\psi)}{(\alpha,\varphi) &  \varphi \vdash \psi} \\[2mm]
\infer[\mathrm{(AND)}]{(\alpha,\varphi \land \psi)}{(\alpha,\varphi) & (\alpha,\psi)} &
\infer[\mathrm{(OR)}]{(\alpha \lor \beta, \varphi)}{(\alpha,\varphi) & (\beta,\varphi)} &
\infer[\mathrm{(CT)}]{(\alpha,\psi)}{(\alpha,\varphi) & (\alpha \land \varphi, \psi)}
\end{tabular}
\end{center}

\begin{center}
	\begin{tabular}{ l l}
		\hline
		$N_i$ & Rules   \\
		\hline	
		$N_{1}$  & ($\top$), (SI), (WO), (AND)   \\
		$N_{2}$  & ($\top$), (SI), (WO), (AND), (OR)  \\
		$N_{3}$  & ($\top$), (SI), (WO), (AND), (CT)  \\
		$N_{4}$  & ($\top$), (SI), (WO),(AND), (OR), (CT) \\
		\hline
	\end{tabular}
\end{center}      
\end{definition}
\subsection{(Proto-)subordination algebras}
\begin{definition}[(Proto-)subordination algebra]
\label{def: subordination algebra} A {\em proto-subordination algebra} is a tuple $\mathbb{S} = (A, \prec)$ such that $A$ is a (possibly bounded) poset (with bottom denoted $\bot$ and top denoted $\top$ when they exist), and $\prec\ \subseteq A\times A$. A  proto-subordination algebra  is named as indicated in the left-hand column in the table below when $\prec$ satisfies the properties indicated in the right-hand column. In what follows, we will refer to a proto-subordination algebra  $\mathbb{S} = (A, \prec)$ as e.g.~{\em (distributive) lattice-based ((D)L-based)}, or {\em Boolean-based (B-based)} if $A$ is a (distributive) lattice, a Boolean algebra, and so on. More in general, for any logic $\mathcal{L}$, we say that $\mathbb{S} = (A, \prec)$ is $\mathsf{Alg}(\mathcal{L})$-{\em based} if $A\in \mathsf{Alg}(\mathcal{L})$. The reader can safely assume that $A$ is a (bounded distributive) lattice, or a Boolean algebra, although, if this is not specified, the results presented below will hold more generally. We will flag out the assumptions we need in the statements of propositions.
\begin{center}
\begin{longtable}{rlrl} 
     ($\bot$) & $\bot \prec \bot$  &
     ($\top$) & $\top \prec \top$ \\
     (SI) & $a \leq b \prec x \Rightarrow a \prec x$ &
     (WO) & $b \prec x \leq y \Rightarrow b \prec y$ \\
     (AND) & $a \prec x \ \& \ a \prec y \Rightarrow a \prec x \wedge y$ &
     (OR) & $a \prec x \ \& \ b \prec x \Rightarrow a \vee b \prec x$ \\
     (D) & $a \prec c \Rightarrow \exists b(a \prec b\ \&\ b \prec c)$ &
     (S6) & $a \prec b \Rightarrow \neg b \prec \neg a$ \\
     (CT) & $a \prec b \ \& \ a \wedge b \prec c \Rightarrow a \prec c$ &
     (T) & $a \prec b \ \& \ b \prec c \Rightarrow a \prec c$ \\
     (DD) & \multicolumn{3}{l}{ $a \prec x_1\ \&\ a \prec x_2 \Rightarrow \exists x (a \prec x\ \&\ x \leq x_1\ \&\ x\leq x_2)$} \\
	 (UD) & \multicolumn{3}{l}{ $a_1 \prec x\ \&\ a_2 \prec x \Rightarrow \exists a (a \prec x\ \&\ a_1 \leq a\ \&\ a_2 \leq a)$} \\
     (S9)& \multicolumn{3}{l}{ $\exists c( c \prec b\ \&\ x\prec a \vee c) \iff \exists a^\prime \exists b^\prime (a^\prime \prec a\ \&\ b^\prime \prec b\ \&\ x \leq a^\prime \vee b^\prime)$
     } \\
     (SL1) & \multicolumn{3}{l}{ $a\prec b \vee c \Rightarrow \exists b^\prime \exists c^\prime (b^\prime \prec b\ \&\ c^\prime \prec c\ \&\ a \prec b^\prime \vee c^\prime)$
     } \\
     (SL2) & \multicolumn{3}{l}{ $b \wedge c \prec a \Rightarrow \exists b^\prime \exists c^\prime (b^\prime \prec b\ \&\ c^\prime \prec c\ \&\ b^\prime \wedge  c^\prime \prec a)$
     }
\end{longtable}
\end{center}
\end{definition}

\begin{center}
	\begin{tabular}{ l l}
		\hline
		Name & Properties  \\
		\hline
		$\Diamond$-premonotone & (SI)   \\
			$\blacksquare$-premonotone & (WO)   \\
			premonotone & (SI) (WO)   \\
			$\Diamond$-directed & (WO) (DD)   \\
			$\blacksquare$-directed & (SI) (UD)   \\
			
			$\Diamond$-monotone & (WO) (DD) (SI)  \\
			$\blacksquare$-monotone & (SI) (UD) (WO)   \\
			directed/monotone & (SI) (WO) (UD) (DD) \\
			$\Diamond$-regular & (SI) (WO) (DD) (OR) \\
			$\blacksquare$-regular & (SI) (WO) (UD) (AND) \\
			regular & (SI) (WO) (OR) (AND)\\
			$\Diamond$-normal & (SI) (WO) (DD) (OR) ($\bot$)  \\
			$\blacksquare$-normal & (SI) (WO) (UD) (AND) ($\top$) \\
			subordination algebra & (SI) (WO) (OR) (AND) ($\bot$) ($\top$)\\
		\hline
	\end{tabular}
\end{center} 

Normative systems can be interpreted in proto-subordination algebras as follows:

\begin{definition}
 A {\em model} for an input/output logic $\mathbb{L} =(\mathcal{L}, N)$ is a tuple $\mathbb{M} = (\mathbb{S}, h)$ s.t.~$\mathbb{S} = (A, \prec)$ is an $\mathsf{Alg}(\mathcal{L})$-based proto-subordination algebra (i.e.~$A\in \mathsf{Alg}(\mathcal{L})$), and $h:\mathrm{Fm}\to A$ is a homomorphism s.t.~for all $\varphi, \psi\in \mathrm{Fm}$, if $(\varphi, \psi)\in N$, then $h(\varphi)\prec h(\psi)$.


\end{definition}

\subsection{Canonical extensions and slanted algebras}
In the present subsection, we adapt material from  \cite[Sections 2.2 and 3.1]{de2020subordination},\cite[Section 2]{DuGePa05}. For any poset $A$, a subset $B\subseteq A$ is {\em upward closed}, or an {\em up-set} (resp.~{\em downward closed}, or a {\em down-set}) if $\lfloor B\rfloor: = \{c\in A\mid \exists b(b\in B\ \&\ b\leq c)\}\subseteq B$ (resp.~$\lceil B\rceil: = \{c\in A\mid \exists b(b\in B\ \&\ c\leq b)\}\subseteq B$); a subset $B\subseteq A$  is {\em down-directed} (resp.~{\em up-directed}) if, for all  $a, b\in B$,  some $x\in B$ exists s.t.~$x\leq a$ and $x\leq b$ (resp.~$a\leq x$ and $b\leq x$). 
It is straightforward to verify that when $A$ is a lattice, down-directed upsets and up-directed down-sets coincide with  lattice filters and ideals, respectively.
\begin{definition}
Let $A$ be a  subposet of a complete lattice $A'$. 
\begin{enumerate}
\item An element $k\in A'$ is {\em closed} if $k = \bigwedge F$ for some down-directed  $F\subseteq A$; an element $o\in A'$ is {\em open} if $o = \bigvee I$ for some up-directed  $I\subseteq A$;  
\item  $A$ is {\em dense} in $A'$ if every element of $A'$ can be expressed both as the join of closed elements and as the meet of open elements of $A$.
\item $A$ is {\em compact} in $A'$ if, for all $F, I\subseteq A$ s.t.~$F$ is  down-directed, $I$ is  up-directed, if $\bigwedge F\leq \bigvee I$ then $a\leq b$ for some $a\in  F$ and $b\in I$.\footnote{When the poset $A$ is a lattice, the compactness can be equivalently reformulated by dropping the requirements that $F$ be down-directed and $I$ be up-directed.} 
\item The {\em canonical extension} of a poset $A$ is a complete lattice $A^\delta$ containing $A$
as a dense and compact subposet.
\end{enumerate}
\end{definition}
The canonical extension $A^\delta$ of any poset $A$ always exists and is  unique up to an isomorphism fixing $A$ (cf.\ \cite[Propositions 2.6 and 2.7]{DuGePa05}).
%
 The set of the closed (resp.~open) elements    of $A^\delta$ is denoted $K(A^\delta)$ (resp.~$O(A^\delta)$). The following proposition collects well known facts which we will use in the remainder of the paper.
 
 \begin{proposition}
 \label{prop:background can-ext}
 For every poset $A$,
 \begin{enumerate}[label=(\roman*)]
     \item if $A$ is a distributive lattice (DL), then $A^\delta$ is completely distributive.
     \item if $\neg: A\to A$ is  antitone and s.t.~$(A, \neg)\models \forall a\forall b(\neg a\leq  b \Leftrightarrow \neg b\leq  a)$, then $\neg^\sigma: A^\delta\to A^\delta$ defined as $\neg^\sigma o: =\bigwedge \{\neg a\mid a\leq o\}$ for any $o\in O(A^\delta)$ and $\neg^\sigma u: =\bigvee \{\neg^\sigma o\mid u\leq o\}$ for any $u\in A^\delta$ is antitone and s.t.~$(A^\delta, \neg^\sigma) \models\forall u\forall v(\neg u\leq v \Leftrightarrow \neg v\leq u)$. If in addition, $(A, \neg)\models a\leq \neg\neg a$, then $(A^\delta,\neg^\sigma)\models u\leq \neg\neg u$. Hence, if $(A, \neg)\models a = \neg\neg a$ (i.e.~$\neg$ is involutive), then  $(A^\delta,\neg^\sigma)\models a = \neg\neg a$.
     \item if $\neg: A\to A$ is  antitone and s.t.~$(A, \neg)\models \forall a\forall b( a\leq \neg b \Leftrightarrow  b\leq \neg a)$, then $\neg^\pi: A^\delta\to A^\delta$ defined as $\neg^\pi k: =\bigvee \{\neg a\mid k\leq a\}$ for any $k\in K(A^\delta)$ and $\neg^\pi u: =\bigwedge \{\neg^\pi k\mid k\leq u\}$ for any $u\in A^\delta$ is antitone and s.t.~$(A^\delta, \neg^\pi)\models\forall u\forall v( u\leq \neg v \Leftrightarrow  v\leq\neg u)$. If in addition, $(A, \neg)\models  \neg\neg a\leq a$, then $(A^\delta,\neg^\pi)\models \neg\neg u\leq  u$. Hence, if $\neg$ is involutive, then so is $\neg^\pi$.
 \end{enumerate}
 \end{proposition}
\begin{proof}
    (i) see \cite[Theorem 2.5]{Gehrke2004BoundedDL}. (ii) For the first part of the statement, see \cite[Proposition 3.6]{DuGePa05}. Let us assume that $(A, \neg)\models a\leq \neg\neg a$, and show that $(A^\delta,\neg^\sigma)\models u\leq \neg\neg u$. The following chain of equivalences holds in $(A^\delta,\neg^\sigma)$, where  $k$ ranges in $K(A^\delta)$ and $o$ in $O(A^\delta)$:
    \begin{center}
        \begin{tabular}{c ll}
             & $\forall u( u\leq \neg\neg u)$ \\
           iff   & $\forall u\forall k\forall o((k\leq u\ \&\ \neg\neg u\leq o) \Rightarrow k\leq o)$ & denseness  \\ iff   & $\forall k\forall o(\exists u(k\leq u\ \&\ \neg\neg u\leq o) \Rightarrow k\leq o)$ &   \\
       iff   & $\forall k\forall o( \neg\neg k\leq o \Rightarrow k\leq o)$ & Ackermann's lemma \\    iff   & $\forall k(  k\leq \neg\neg k)$. & denseness \\
        \end{tabular}
    \end{center}
    Hence, to complete the proof, it is enough to show that, if $k\in K(A^\delta)$, then $k\leq \neg\neg k$. By definition, $k = \bigwedge D$ for some down-directed $D\subseteq A$. Since $\neg^\sigma$ is a (contravariant) left adjoint, $\neg^\sigma$ is completely meet-reversing. Hence, $\neg k = \neg(\bigwedge D) = \bigvee \{\neg d\mid d\in D\}$, and since $D$ being down-directed implies that $\{\neg d\mid d\in D\}\subseteq A$ is up-directed, we deduce that $\neg k\in O(A^\delta)$. Hence,
    \begin{center}
        \begin{tabular}{rcll}
              $\neg\neg k$
             & $=$&$ \bigwedge \{\neg a\mid a\leq \neg k\}$ &$\neg k\in O(A^\delta)$ \\
             & $=$&$  \bigwedge \{\neg a\mid a\leq \bigvee \{\neg d\mid d\in D\}\}$\\ 
             & $=$&$  \bigwedge \{\neg a\mid \exists d(d\in D\ \&\ a\leq \neg d )\}$. & compactness\\ \end{tabular}
    \end{center}
    Hence, to show that $\bigwedge \lfloor D\rfloor = \bigwedge D = k\leq \neg\neg k$, it is enough to show that if $a\in A$ is s.t.~$\exists d(d\in D\ \&\ a\leq \neg d )$, then $d'\leq \neg a$ for some $d'\in D$. From $a\leq \neg d$, by the antitonicity of $\neg$, it follows $\neg\neg d\leq \neg a$; combining this inequality with $d\leq \neg\neg d$ which holds by assumption for all $d\in A$, we get  $d': = d\leq \neg a$, as required. Finally, notice that by instantiating the left-hand inequality in the equivalence $(A^\delta, \neg^\sigma) \models\forall u\forall v(\neg u\leq v \Leftrightarrow \neg v\leq u)$ with $v: = \neg u$, one immediately gets $(A^\delta, \neg^\sigma) \models\forall u(\neg \neg u\leq u)$.  (iii) dual to (ii).
\end{proof}

\begin{definition}
A {\em slanted algebra} is a triple $\mathbb{A} = (A, \Diamond, \blacksquare)$ such that $A$ is a poset, and $\Diamond, \blacksquare: A\to A^\delta$   s.t.~$\Diamond a \in K(A^\delta)$  and $\blacksquare a \in O(A^\delta)$ for every $a$. A slanted algebra as above is {\em tense} if $\Diamond a\leq b$ iff $a\leq \blacksquare b$ for all $a, b\in A$; 
is {\em monotone} if $\Diamond$ and $\blacksquare$ are monotone; is {\em regular} if $\Diamond$ and $\blacksquare$ are regular (i.e.~$\Diamond(a\vee b) = \Diamond a \vee \Diamond b$ and $\blacksquare(a\wedge b) = \blacksquare a \wedge \blacksquare b$ for all $a, b\in A$); is {\em normal} if $\Diamond$ and $\blacksquare$ are normal (i.e.~they are regular and $\Diamond\bot = \bot$  and $\blacksquare\top = \top$).

\end{definition}
The following definition is framed in the context of monotone slanted algebras, but can be given for arbitrary slanted algebras, albeit at the price of complicating the definition of $\Diamond^\sigma$ and $ \blacksquare^\pi$. Because we are mostly going to apply it in the monotone setting, we present the simplified version here.
\begin{definition}
\label{def: sigma and pi extensions of slanted}
For any monotone slanted algebra $\mathbb{A} = (A, \Diamond, \blacksquare)$ the {\em canonical extension} of $\mathbb{A}$ is the (standard!) modal algebra $\mathbb{A}^\delta: = (A^\delta, \Diamond^\sigma, \blacksquare^\pi)$ such that $\Diamond^\sigma, \blacksquare^\pi: A^\delta\to A^\delta$ are defined as follows:  for every $k\in K(A^\delta)$, $o\in O(A^\delta)$ and $u\in A^\delta$,
\[\Diamond^\sigma k:= \bigwedge\{ \Diamond a\mid a\in A\mbox{ and } k\leq a\}\quad \Diamond^\sigma u:= \bigvee\{ \Diamond^\sigma k\mid k\in K(A^\delta)\mbox{ and } k\leq u\}\] 
\[\blacksquare^\pi o:= \bigvee\{ \blacksquare a\mid a\in A\mbox{ and } a\leq o\},\quad \blacksquare^\pi u:= \bigwedge\{ \blacksquare^\pi o\mid o\in O(A^\delta)\mbox{ and } u\leq o\}.\]
\end{definition}

 For any slanted algebra $\bba$, any  assignment $v:\mathsf{PROP}\to \bba$ uniquely extends to a homomorphism $v: \mathcal{L}\to \bbas$ (abusing notation, the same symbol denotes   both the assignment and its homomorphic extension).  Hence,
\begin{definition} 
\label{def:slanted satisfaction and validity}
 A modal inequality $\phi\leq\psi$ is {\em satisfied} in a slanted algebra $\bba$ under the assignment $v$ (notation: $(\bba, v)\models \phi\leq\psi$) if $(\bbas, e\cdot v)\models \phi\leq\psi$ in the usual sense, where $e\cdot v$ is the assignment on $\bbas$ obtained by composing the canonical embedding $e: \bba\to \bbas$ to the assignment $v:\mathsf{Prop}\to \bba$. 
 
 Moreover, $\phi\leq\psi$ is {\em valid} in $\bba$ (notation: $\bba\models \phi\leq\psi$) if $(\bbas, e\cdot v)\models \phi\leq\psi$ for every assignment $v$ into $\bba$  (notation: $\bbas\models_{\bba} \phi\leq\psi$). 
\end{definition}

\section{Proto-subordination algebras and slanted algebras}
\label{sec: proto and slanted}
\noindent Let $\mathbb{S} = (A, \prec)$  be a proto-subordination algebra s.t.~$\mathbb{S}\models \mathrm{(DD)+(UD)}$. The slanted algebra associated with  $\mathbb{S}$ is  $\mathbb{S}^* = (A, \Diamond, \blacksquare)$ s.t.~$\Diamond a \coloneqq \bigwedge {\prec}[a]$ and $\blacksquare a \coloneqq \bigvee {\prec^{-1}}[a]$ for any $a$. From $\mathbb{S}\models \mathrm{(DD)}$ it follows that ${\prec}[a]$ is down-directed for every $a\in A$, hence $\Diamond a\in K(A^\delta)$. Likewise, $\mathbb{S}\models \mathrm{(UD)}$ guarantees that $\blacksquare a\in O(A^\delta)$ for all $a\in A$.

\begin{lemma} \label{lem: diamond-output equivalence}
For any proto-subordination algebra $\mathbb{S} = (A, \prec)$ and all $a, b\in A$,
\begin{enumerate}[label=(\roman*)]
    \item $a \prec b$ implies $\Diamond a \leq b$ and $a \leq \blacksquare b$.
    \item if $ \mathbb{S}\models \mathrm{(WO) +  (DD)}$, then $\Diamond a\leq b$ iff $a\prec b$.
    \item if $\mathbb{S}\models \mathrm{ (SI) + (UD)}$, then $a\leq \blacksquare b$ iff $a\prec b$.
\end{enumerate}
\end{lemma}
\begin{proof}
    (i) $a\prec b$ iff $b\in {\prec}[a]$ iff $a \in {\prec}^{-1}[b]$, hence $a\prec b$ implies $b\geq \bigwedge {\prec}[a] = \Diamond a$ and $a \leq \bigvee{\prec}^{-1}[b] = \blacksquare b$.
    
    \noindent (ii) By (i), to complete the proof, we need to show the `only if' direction. The assumption $ \mathbb{S}\models \mathrm{(DD)}$ implies that  ${\prec}[a]$ is down-directed for any $a\in A$. Hence,  by compactness,  $\bigwedge {\prec}[a] = \Diamond a\leq b$ implies  that $c \leq b$ for some $c \in {\prec}[a]$, i.e.~$a \prec c \leq b$ for some $ c \in A$, and by (WO), this implies that $a\prec b$, as required.
    (iii) is proven similarly, by observing that $\mathbb{S}\models \mathrm{(UD)}$ implies that ${\prec}^{-1}[a]$ is up-directed for every $a\in A$.
\end{proof}

\begin{lemma}
For any lattice-based proto-subordination algebra $\mathbb{S} = (A, \prec)$,
\begin{enumerate}[label=(\roman*)]
    \item $\mathbb{S}\models\mathrm{(OR)}$ implies $\mathbb{S}\models\mathrm{(UD)}$.
    \item $\mathbb{S}\models\mathrm{(AND)}$ implies $\mathbb{S}\models\mathrm{(DD)}$.
    \item If $\mathbb{S}\models\mathrm{ (SI)}$, then $\mathbb{S}\models\mathrm{(UD)}$ iff $\mathbb{S}\models\mathrm{(OR)}$.
    \item if $\mathbb{S}\models\mathrm{ (WO)}$, then $\mathbb{S}\models\mathrm{ (DD)}$ iff $\mathbb{S}\models\mathrm{ (AND)}$.
\end{enumerate}
\end{lemma}
\begin{proof}
    (i) and (ii) are straightforward. As for (iii), by (i), to complete the proof we need to show the `only if' direction. Let $a, b, x\in A$ s.t.~$a \prec x$ and $b \prec x$. By (UD), this implies that $c\prec x$ for some $c\in A$ such that $a\leq c$ and $b\leq c$. Since $A$ is a lattice, this implies that $a\vee b\leq c\prec x$, and by (SI), this implies that $a\vee b\prec x$, as required. 
   (vi) is proven similarly.
\end{proof}

\begin{lemma} \label{lem: rleation to axioms}
For every proto-subordination algebra $\mathbb{S} = (A, \prec)$,
\begin{enumerate}[label=(\roman*)]
   
    \item If  $\mathbb{S}\models\mathrm{ (SI)}$, then: 
    \begin{enumerate}
    \item $\Diamond$  on $\mathbb{S}^*$ is monotone;
    \item if $\mathbb{S}$ is DL-based, then   $\mathbb{S}\models\mathrm{ (AND)}$ implies $\mathbb{S}^\ast\models\blacksquare a \wedge \blacksquare b\leq \blacksquare (a \wedge b)$;
    \item if   
    $\mathbb{S}\models\mathrm{(UD)}$, then $\mathbb{S}\models\mathrm{ (AND)}$ implies $\mathbb{S}^\ast\models\blacksquare a \wedge \blacksquare b\leq \blacksquare (a \wedge b)$.
    \end{enumerate}
    \item If  $\mathbb{S}\models\mathrm{ (WO)}$, then \begin{enumerate}
    \item $\blacksquare$ on $\mathbb{S}^*$ is monotone;
    
    \item if 
    $\mathbb{S}$ is DL-based, then  $\mathbb{S}\models\mathrm{ (OR)}$ implies $\mathbb{S}^\ast\models\Diamond (a\vee b)\leq \Diamond a \vee \Diamond b$;

    \item if   
    $\mathbb{S}\models\mathrm{(DD)}$, then  $\mathbb{S}\models\mathrm{ (OR)}$ implies $\mathbb{S}^\ast\models\Diamond (a\vee b)\leq \Diamond a \vee \Diamond b$.
    \end{enumerate}
    
    \item If   $\mathbb{S}\models\mathrm{(\bot)}$, then $\mathbb{S}^\ast\models \Diamond\bot \leq \bot$.
    \item If   $\mathbb{S}\models\mathrm{(\top)}$, then  $\mathbb{S}^\ast\models\top\leq \blacksquare\top$.
\end{enumerate}
\end{lemma}

\begin{proof}
	(i)(a) Let $a, b\in A$ s.t.~$a \leq b$. To show that $\Diamond a = \bigwedge{\prec}[a]\leq \bigwedge{\prec}[b] = \Diamond b$, it is enough to show that ${\prec}[b]\subseteq {\prec}[a]$, i.e.~that if $x\in A$ and $b\prec x$, then $a\prec x$. Indeed, by (SI), $a\leq b\prec x$ implies $a\prec x$, as required.  (ii) (a) is shown similarly.

 	(ii)(b) Let $a, b\in A$. By definition,  $\Diamond (a \vee b) = \bigwedge {\prec} [a \vee b] = \bigwedge\{ d \mid a\vee b \prec d\}$, and, since $A^\delta$ is completely distributive when $A$ is a DL (cf.~Proposition \ref{prop:background can-ext}(i)),
 	\begin{center}
 	    \begin{tabular}{c ll l}
 	    $\Diamond a \vee \Diamond b $ & $=$ & $(\bigwedge {\prec} [a])\vee(\bigwedge {\prec} [b])$\\
 	    & $=$ & $ \bigwedge \{ c \vee c^\prime \mid a \prec c$ and $b \prec c^\prime \}$.\\
 	    \end{tabular}
 	\end{center}
  So, to show that $\Diamond (a\vee b)\leq \Diamond a \vee \Diamond b$, it is enough to show that $\{c \vee c^\prime \mid a \prec c$ and $b \prec c^\prime\}\subseteq \{ d \mid a\vee b \prec d\} $, i.e.~that for all $c, c'\in A$, if $a \prec c$ and $b\prec c'$, then $a \vee b \prec c  \vee c^\prime$. By  (WO),  $a \prec c \leq c\vee c'$ and $b\prec c^\prime\leq c\vee c'$ imply that $a \prec c \vee c^\prime $ and $b\prec c\vee c^\prime$, which by (OR) implies that $a \vee b \prec c  \vee c^\prime$, as required. (i)(b) is argued similarly.
 (ii)(c) To show that $\Diamond (a\vee b)\leq \Diamond a \vee \Diamond b$, it is enough to show that for any $x\in A$, if $\Diamond a \vee \Diamond b\leq x$, then $\Diamond (a\vee b)\leq x$.
    
\begin{center}
\begin{tabular}{rcll}
$\Diamond a \vee \Diamond b \leq x$ &
iff & $\Diamond a \leq x$ and $\Diamond b \leq x$  \\
&iff & $a \prec x$ and $b \prec x$ & Lemma \ref{lem: diamond-output equivalence} (ii) (WO) + (DD)\\
& implies & $a \vee b \prec x$ &  (OR)\\
& implies & $\Diamond (a \vee b) \leq x$ & Lemma \ref{lem: diamond-output equivalence} (i)\\
\end{tabular}
\end{center}
(i)(c) is proven similarly.   
  (iii) By assumption, $\bot \prec \bot $, i.e.~$\bot \in {\prec}[\bot]$, which implies $\Diamond \bot =\bigwedge {\prec} [\bot]\leq \bot$, as required. (iv) is argued similarly.
    \end{proof}

The following lemma gives a converse of Lemma \ref{lem: rleation to axioms} for $\Diamond$-directed or $\blacksquare$-directed proto-subordination algebras.

\begin{lemma}\label{lem: axioms to relation}
For any proto-subordination algebra $\mathbb{S} =(A, \prec)$, 
\begin{enumerate}[label=(\roman*)]
    \item If $\mathbb{S}\models\mathrm{(WO)} + \mathrm{(DD)}$,  then:
    \begin{enumerate}
        \item $\mathbb{S}\models\mathrm{(SI)}\quad$ iff $\quad \Diamond$ on $\mathbb{S}^\ast$ is monotone.
        \item $\mathbb{S}\models\mathrm{(OR)}\quad $ iff $\quad \mathbb{S}^\ast\models \Diamond (a \vee b) \leq \Diamond a \vee \Diamond b$.
        \item $\mathbb{S}\models(\bot)\quad $ iff $\quad\mathbb{S}^\ast\models\Diamond \bot \leq \bot$.
    \end{enumerate}
    \item If $\mathbb{S}\models\mathrm{(SI)} + \mathrm{(UD)}$, then:
    \begin{enumerate}
        \item $\mathbb{S}\models\mathrm{(WO)}\quad $ iff $\quad\blacksquare$ on $\mathbb{S}^\ast$ is monotone;
        \item $\mathbb{S}\models\mathrm{(AND)}\quad $ iff $\quad\mathbb{S}^\ast\models\blacksquare a \wedge \blacksquare b \leq \blacksquare (a \wedge b)$; 
        \item $\mathbb{S}\models(\top)\quad $ iff $\quad\mathbb{S}^\ast\models\top  \leq \blacksquare\top$.
    \end{enumerate}
\end{enumerate}
\end{lemma}

\begin{proof}
   We only show the items in  (i), the proofs of those in (ii) being similar. (a) By Lemma \ref{lem: rleation to axioms} (i)(a), the proof is complete if we show the `if' direction.  Let $a, b, x\in A$ s.t.~$a\leq b\prec x$. By Lemma \ref{lem: diamond-output equivalence} (ii), to show that $a\prec x$,  it is enough to show that $\Diamond a\leq x$. Since $\Diamond$ is monotone, $a\leq b$ implies $\Diamond a\leq \Diamond b$, and, again by Lemma \ref{lem: diamond-output equivalence} (ii), $b\prec x$ implies that $\Diamond b\leq x$. Hence, $\Diamond a\leq x$, as required.
    
    (b) By Lemma \ref{lem: rleation to axioms} (ii)(c), the proof is complete if we show the `if' direction. Let $a, b, x\in A$ s.t.~$a\prec x$ and $b\prec x$. By Lemma \ref{lem: diamond-output equivalence} (ii), to show that $a\vee b\prec x$,  it is enough to show that $\Diamond (a\vee b)\leq x$, and since $\mathbb{S}^\ast\models \Diamond (a \vee b) \leq \Diamond a \vee \Diamond b$, it is enough to show that $\Diamond a\vee\Diamond b\leq x$, i.e.~that $\Diamond a\leq x$ and $ \Diamond b\leq x$. These two inequalities hold by Lemma \ref{lem: diamond-output equivalence} (ii), and the assumptions on $a, b$ and $x$.
    
    (c)  By Lemma \ref{lem: diamond-output equivalence} (ii),  $\bot\prec \bot$ is equivalent to $\Diamond \bot\leq \bot$, as required.
\end{proof}
\begin{corollary}
\label{cor: charact monotone reg norm}
For every directed proto-subordination algebra $\mathbb{S} = (A, \prec)$,
\begin{enumerate}
   
    \item $\mathbb{S}$ is monotone iff $\mathbb{S}^*$ is monotone;
    \item $\mathbb{S}$ is regular iff $\mathbb{S}^*$ is regular;
    \item $\mathbb{S}$ is a subordination algebra iff $\mathbb{S}^*$ is normal.
    \end{enumerate}
\end{corollary}

\begin{lemma} \label{lem: diamond-output equivalence extended}
 For any proto-subordination algebra $\mathbb{S} = (A, \prec)$, for all $a, b\in A$, $k \in K(A^\delta)$, and $o \in O(A^\delta)$, and all $D, U\subseteq A$,
\begin{enumerate}[label=(\roman*)]
    \item if  $\mathbb{S}\models\mathrm{(SI) +(DD)+ (WO)}$, then 
    \begin{enumerate}
        \item if $D\subseteq A$ is down-directed, then so is ${\prec}[D]\coloneqq \{c\ |\ \exists a (a \in  D\ \&\ a \prec c)\}$;
        \item if $k = \bigwedge D$ for some down-directed $D \subseteq A$, then $\Diamond k = \bigwedge {\prec}[D] \in K(A^\delta)$;
        \item $\Diamond k \leq b$ implies $a\prec b$ for some $a \in A$ s.t.~$ k \leq a$.
         \item $\Diamond k \leq o$ implies $a\prec b$ for some $a,b \in A$ s.t.~$ k \leq a$ and $b\leq o$.
    \end{enumerate}
    \item if  $\mathbb{S}\models\mathrm{(WO)+(UD)+(SI)}$, then 
    \begin{enumerate}
        \item if $U\subseteq A$ is up-directed, then so is ${\prec}^{-1}[U]\coloneqq \{c\ |\ \exists a (a \in  U\ \&\ c \prec a)\}$;
        \item if $o = \bigvee U$ for some up-directed $U \subseteq A$, then $\blacksquare o = \bigvee {\prec}^{-1}[U] \in O(A^\delta)$;
        \item $a \leq \blacksquare o$ implies $a\prec b$ for some $b \in A$ s.t.~$ b \leq o$.
        \item $k \leq \blacksquare o$ implies $a\prec b$ for some $a,b \in A$ s.t.~$k\leq a$ and $ b \leq o$.
    \end{enumerate}
\end{enumerate}
\end{lemma}
\begin{proof}
  We only prove  (i), the proof of (ii) being similar. 
  
   (a)  If $c_i  \in {\prec}[D]$ for $1\leq i\leq 2$, then $a_i \prec c_i$  for some $a_i \in D$. Since $D$ is down-directed, some $a \in D$ exists s.t.~$a \leq a_i$ for each $i$. Thus,  $\mathrm{(SI)}$ implies that $a \prec c_i$, from which the claim  follows by $\mathrm{(DD)}$.
  
  (b) By definition, $\Diamond k = \bigwedge \{ \Diamond a\ |\ a \in A, k \leq a\} = \bigwedge \{c\ |\ \exists a (a \prec c\ \&\  k \leq a)\}$. Since $k = \bigwedge D$ for some $D \subseteq A$ down-directed, by compactness, $k \leq a$ implies $d \leq a$ for some $d \in D$, thus $\Diamond k = \bigwedge \coloneqq \{c\ |\ \exists a (a \prec c\ \&\  k \leq a)\} = \bigwedge \{c\ |\ \exists a (a \in \lfloor D\rfloor\ \&\ a \prec c)\} = \bigwedge {\prec}[D]\in K(A^\delta)$, the last membership holding by (a). 
  
  (c) By (b),  $\Diamond k\in K(A^\delta)$. Hence, $\Diamond k\leq b$ implies by compactness that $c \leq b$ for some $c\in A$ s.t.~$a \prec c$ for some $a \in  D$ (hence $k = \bigwedge D\leq a$). By $\mathrm{(WO)}$, this implies that $a \prec b$ for some $a \in A$ s.t.~$k \leq a$, as required.
  
  (d) By (b), $\Diamond k\in K(A^\delta)$. Since $o\in O(A^\delta)$, some updirected $U\subseteq A$ exists s.t.~$o = \bigvee U$. Hence, by compactness, $\Diamond k\leq o$ implies that $a\prec b$ for some $a\in A$ s.t.~$k\leq a$ and some $b\in U$ (for which $b\leq o$).
\end{proof}

\begin{proposition}
\label{prop:characteriz}
For any proto-subordination algebra $\mathbb{S} = (A, \prec)$,  
\begin{enumerate}[label=(\roman*)]
   \item  $\mathbb{S}\models\; \prec\ \subseteq\ \leq\ $ iff $\ \mathbb{S}^*\models a\leq \Diamond a\ $ iff $\ \mathbb{S}^*\models \blacksquare a\leq  a$.
    \item  If $\mathbb{S}\models\mathrm{(WO)}+\mathrm{(DD)}$, then
    $\; \mathbb{S}\models\;\leq\ \subseteq\ \prec\quad $ iff $\quad \mathbb{S}^*\models\Diamond a\leq a$;
    \item  if  
    $\mathbb{S}\models\mathrm{(WO)}+\mathrm{(DD)}+\mathrm{(SI)}$, then %
    \begin{enumerate}
    \item $\mathbb{S}\models \mathrm{(T)}\quad $  iff $\quad \mathbb{S}^*\models\Diamond a\leq \Diamond \Diamond a$.
    \item  
    $\mathbb{S}\models \mathrm{(D)}\quad $  iff $\quad \mathbb{S}^*\models\Diamond\Diamond a\leq  \Diamond a$.
    \end{enumerate}
    \item if $\mathbb{S}\models  \mathrm{(WO)}+ 
    \mathrm{(DD)} + \mathrm{(SI)}$ and 
    is meet-semilattice based, then 
    \begin{enumerate}
    \item $\mathbb{S}\models \mathrm{(CT)}\ $  iff $\ \mathbb{S}^*\models\Diamond a\leq \Diamond (a\wedge \Diamond a)$.
    \item $\mathbb{S}\models \mathrm{(SL2)}\ $  iff $\ \mathbb{S}^*\models \Diamond (\Diamond a \wedge \Diamond b) \leq \Diamond (a \wedge b)$.\end{enumerate}
    \item if $\mathbb{S}\models\mathrm{(SI)}$, then $\mathbb{S}\models \mathrm{(CT)}$ implies  $\mathbb{S}\models \mathrm{(T)}$.

              \item  if $\mathbb{S}$ is directed and based on $(A, \neg)$ with $\neg$ antitone,  involutive, and (left or right) self-adjoint,
              \begin{enumerate}
              \item $\mathbb{S}\models \mathrm{(S6)}\ $ iff  $\ \mathbb{S^*}\models \neg \Diamond a = \blacksquare \neg a$, thus $\blacksquare a \coloneqq \neg \Diamond \neg a$.
              \item $\mathbb{S}\models \mathrm{(S6)}\ $ iff  $\ \mathbb{S^*}\models  \Diamond \neg a =  \neg\blacksquare  a$, thus $\Diamond a \coloneqq \neg \blacksquare \neg a$.
              \end{enumerate}
			\item If  $\mathbb{S}\models\mathrm{(SI)+(UD)+(WO)}$ and 
			is join-semilattice based, then 
			\begin{enumerate}
			\item $\mathbb{S}\models \mathrm{(S9 \Rightarrow)}  \ $  iff $\ \mathbb{S^*}\models \blacksquare (a \vee \blacksquare b) \leq \blacksquare a \vee \blacksquare b$.
			\item $\mathbb{S}\models \mathrm{(S9 \Leftarrow) }\ $  iff $\ \mathbb{S^*}\models \blacksquare a \vee \blacksquare b\leq \blacksquare (a \vee \blacksquare b)$.
            \item  
            $\mathbb{S}\models \mathrm{(SL1)}\ $  iff $\ \mathbb{S}^*\models\blacksquare (a \vee b)\leq \blacksquare(\blacksquare a \vee \blacksquare b) $. 
            \end{enumerate}
\end{enumerate}
\end{proposition}
\begin{proof}
(i) By definition, $\forall a(a\leq \Diamond a)$ iff $\forall a (a\leq \bigwedge {\prec}[a])$ iff $\forall a (a\leq   \bigwedge \{b\in A\mid a\prec b\})$ iff $\forall a\forall b(a\prec b\Rightarrow a\leq b)$ iff $\prec\; \subseteq\; \leq$. The second part of the statement is proved similarly.

(ii) By  Lemma \ref{lem: diamond-output equivalence} (i), if $a\prec a$, then $\Diamond a\leq a$. Hence, the left-to-right direction follows from the reflexivity of $\leq$ and  the assumption. Conversely,  $\Diamond a\leq a$ and $a\leq b$ imply $\Diamond a\leq b$, which, by Lemma \ref{lem: diamond-output equivalence} (ii) and $\mathbb{S}\models\mathrm{(WO)}+\mathrm{(DD)}$, is equivalent to $a\prec b$.

(iii)(a) From  left to right,  
\begin{center}
    \begin{tabular}{rcll}
  $ \Diamond \Diamond a$       & $=$ & $\bigwedge\{\Diamond b\mid \Diamond a\leq b\}$ & Definition \ref{def: sigma and pi extensions of slanted} applied to $\Diamond a\in K(A^\delta)$ \\
     & $=$ &  $\bigwedge\{\Diamond b\mid  a\prec b\}$  & Lemma \ref{lem: diamond-output equivalence} (ii) since $\mathbb{S}\models\mathrm{(WO)+(DD)}$\\
   & $=$ &  $\bigwedge\{c\mid \exists b( a\prec b\ \&\  b\prec c)\}$ & $\Diamond b = \bigwedge \{c\in A\mid b\prec c\}$ \\
    \end{tabular}
\end{center}
Hence, to show that $\Diamond a = \bigwedge \{c\mid a\prec c\} \leq \Diamond\Diamond a$, it is enough to show that $\{c\mid \exists b( a\prec b\ \&\  b\prec c)\}\subseteq \{c\mid a\prec c\}$, which is immediately implied by the assumption (T).
Conversely, let $a, b, c\in A$ s.t.~$a\prec b$ and $b\prec c$. To show that $a\prec c$, by Lemma \ref{lem: diamond-output equivalence} (ii), and $\mathbb{S}\models\mathrm{(WO)}+\mathrm{(DD)}$,  it is enough to show that $\Diamond a\leq c$, and since $\Diamond a\leq \Diamond\Diamond a$, it is enough to show that $\Diamond\Diamond a\leq c$.
 The assumption $a\prec b$ implies $\Diamond a \leq b $ which implies $\Diamond\Diamond a\leq \Diamond b$, by the monotonicity of $\Diamond$ (which depends on (SI), cf.~Lemma \ref{lem: rleation to axioms}(i)(a)). Hence, combining the latter inequality  with  $\Diamond b \leq c$ (which is implied by $b\prec c$), by the transitivity of $\leq$, we get
$\Diamond\Diamond a\leq c $, as required.

(iii)(b) From left to right, by the definitions spelled out in the proof of (ii)(b), it is enough to show that $ \{c\mid a\prec c\}\subseteq\{c\mid \exists b( a\prec b\ \&\  b\prec c)\}$, which is immediately implied by the assumption (D). 
Conversely, let $a, c\in A$ s.t.~$a\prec c$, and let us  show that $a\prec b$ and $b\prec c$ for some   $b\in A$. 
 The assumption $a\prec c$ implies $\Diamond a \leq c$. Since  $ \Diamond\Diamond a\leq \Diamond a$, this implies $\Diamond\Diamond a\leq c$, i.e.~(see discussion above) $ \bigwedge\{d\in A\mid \exists b( a\prec b\ \&\  b\prec d)\}\leq c$. 
 
 We claim that $D: = \{d\in A\mid \exists b( a\prec b\ \&\  b\prec d)\}$ is down-directed: indeed, if $d_1, d_2\in A$ s.t.~$\exists b_i( a\prec b_i\ \&\  b_i\prec d_i)$ for $1\leq i\leq 2$, then by (DD), some $b\in A$ exists s.t.~$a\prec b$ and $b\leq b_i$. By (SI), $b\leq b_i\prec d_i$ implies $b\prec d_i$. By (DD) again, this implies that some $d\in A$ exists s.t.~$b\prec d$ and $d\leq d_i$, which concludes the proof of the claim.
 
 By compactness, $d\leq c$ for some $d\in  A$ s.t.~$a\prec b$ and $b\prec d$ for some $b\in A$. To finish the proof, it is enough to show that $b\prec c$, which is immediately implied by $b\prec d\leq c$ and (SI).
 
The proofs of the remaining items are collected in Appendix \ref{appendix:proof:prop: charact}. \end{proof}

\section{Applications}
\label{sec: applications}
In the present section, we discuss two independent but connected ways of using the characterization results of the previous section.  Firstly, the output operators $out^N_i$ for $1\leq i\leq 4$ associated with a given input/output logic $\mathbb{L} = (\mathcal{L}, N)$ can be given semantic counterparts in the environment of proto-subordination algebras as follows: for every proto-subordination algebra $\mathbb{S} = (A, \prec)$,  we let $\mathbb{S}_i: = (A, \prec_i)$ where ${\prec_i}\subseteq A\times A$ is the smallest extension   of $\prec$ which satisfies the properties indicated in the following table:
\begin{center}
	\begin{tabular}{ l l}
		\hline
		$\prec_i$ & Properties   \\
		\hline	
		$\prec_{1}$  & ($\top$), (SI), (WO), (AND)   \\
		$\prec_{2}$  & ($\top$), (SI), (WO), (AND), (OR)  \\
		$\prec_{3}$  & ($\top$), (SI), (WO), (AND), (CT)  \\
		$\prec_{4}$  & ($\top$), (SI), (WO), (AND), (OR), (CT) \\
		\hline
	\end{tabular}
\end{center}    
Then, for each $1\leq i\leq 4$, and every $B\subseteq A$, if $k = \bigwedge B\in K(A^\delta)$, then\footnote{When $\mathcal{L}$ does not have the conjunction property, this construction works only under the additional assumption that $B$ is down-directed; however, in most common cases (e.g.~when $\mathbb{S}$ is lattice-based) this assumption is not needed.} \[\Diamond_i^{\sigma}k: = \bigwedge\{ {\prec_i}[a]\mid a\in A \text{ and } k\leq a\}\] encodes the algebraic  counterpart of $out^N_i(\Gamma)$ for any $\Gamma\subseteq \mathrm{Fm}$, and the characteristic properties of $\Diamond_i$ for each $1\leq i\leq 4$ are those  identified in Lemma  \ref{lem: axioms to relation}, Corollary \ref{cor: charact monotone reg norm}, and  Proposition \ref{prop:characteriz}. For any directed proto-subordination algebra $\mathbb{S} = (A, \prec)$, let $\mathbb{S}_i^\ast: = (A, \Diamond_i, \blacksquare_i)$ denote the slanted algebras associated with $\mathbb{S}_i = (A, \prec_i)$ for each $1\leq i\leq 4$. 
\begin{proposition}
\label{prop:output}
For any directed proto-subordination algebra $\mathbb{S} = (A, \prec)$,
\begin{enumerate}
    \item $\Diamond_1$ is the largest monotone map dominated by $\Diamond$ (i.e.~pointwise-smaller than or equal to $\Diamond$), and $\blacksquare_1$ is the largest monotone map dominated by $\blacksquare$.
    \item $\Diamond_2$ is the largest regular map dominated by $\Diamond$, and $\blacksquare_2$ is the largest regular map dominated by $\blacksquare$.
    \item $\Diamond_3$ is the largest monotone map satisfying $\Diamond_3 a\leq \Diamond_3(a\wedge \Diamond_3 a)$ dominated by $\Diamond$, and $\blacksquare_3$ is the largest monotone map satisfying $\blacksquare_3 (a \vee \blacksquare_3 a) \leq \blacksquare_3 a$ dominated by $\blacksquare$.
    \item $\Diamond_4$ is the largest regular map satisfying $\Diamond_4 a\leq \Diamond_4(a\wedge \Diamond_4 a)$ dominated by $\Diamond$, and $\blacksquare_4$ is the largest regular map satisfying $\blacksquare_4 (a \vee \blacksquare_4 a) \leq \blacksquare_4 a$ dominated by $\blacksquare$.
\end{enumerate}
\end{proposition}
\begin{proof}
 By Lemma \ref{lem: axioms to relation} and Proposition \ref{prop:characteriz}, the properties stated in each item of the statement hold for $\Diamond_i$ and $\blacksquare_i$.  To complete the proof, we need to argue for $\Diamond_i$ being the largest such map (the proof for $\blacksquare_i$ is similar).  By Lemma \ref{lem: diamond-output equivalence} (ii), $a\prec_i b$ iff $\Diamond_i a\leq b$  for  all $a, b\in A$ and $1\leq i\leq 4$. Any $f: A\to A^\delta$  s.t.~$f(a)\in K(A^\delta)$ for every $a\in A$ induces a proto-subordination relation  $\prec_f\subseteq A\times A$ defined as $a\prec_f b$ iff $f(a)\leq b$. Clearly, if $f(a)\leq f' (a)$ for every $a\in A$, then ${\prec}_{f'}\subseteq {\prec_f}$. Moreover, if $f(a)< f' (a)$, then, by denseness, $f(a)\leq b$ for some $b\in A$ s.t.~$f'(a)\nleq b$, hence ${\prec}_{f'}\subset {\prec_f}$. 
 
 If $\Diamond_i$ is not the largest map endowed with the properties mentioned in the statement and dominated by $\Diamond$, then  a map $f$ exists which is endowed with these properties such that  $\Diamond_i a \leq f(a) \leq \Diamond a$ for all $a \in A$, and  $\Diamond_i b < f(b)$ for some $b \in A$. Then, by the argument in the previous paragraph,  ${\prec} = {\prec_\Diamond} \subseteq {\prec_f} \subset {\prec}_{\Diamond_i}=\prec_i$. As $f$ is endowed with the the properties mentioned in the statement,  $\prec_f$ is an extension of $\prec$ which enjoys the required properties, and is strictly contained in $\prec_i$. Hence, $\prec_i$ is not the smallest such extension. 
\end{proof}

 As to the second application, in \cite{celani2020subordination}, Celani introduces an  expansion of Priestley's duality for bounded distributive lattices to
 {\em subordination lattices}, i.e.~tuples $\mathbb{S} =(A, \prec)$ such that $A$ is a distributive lattice and ${\prec}\subseteq A\times A$ is a subordination relation.\footnote{In the terminology of the present paper, subordination lattices are subordination algebras based on  bounded distributive lattices (cf.~Definition \ref{def: subordination algebra}).} The  dual structure of  any subordination lattice $\mathbb{S} =(A, \prec)$ is referred to as the {\em (Priestley) subordination space} of $\mathbb{S}$, and is defined as  $\mathbb{S}_*: = (X(A), R_\prec)$, where $X(A)$ is (the Priestley space dual to $A$, based on) the set of prime filters of $A$,   and $R_\prec \subseteq X(A) \times X(A)$ is defined as follows: for all prime filters $P, Q$ of $A$, 
\[
(P,Q) \in R_\prec \quad \text{ iff }\quad {\prec}[P]: =\{x\in A\mid \exists a(a\in P\ \&\ a\prec x)\} \subseteq Q.
\]

Up to isomorphism, we can equivalently define the subordination space of $\mathbb{S}$ as follows:
\begin{definition}
 \label{def: subordination space in adelta}
 The  {\em subordination space} associated with a subordination lattice $\mathbb{S} =(A, \prec)$ is $\mathbb{S}_*: = (\jty(A^\delta), R_\prec)$, where $\jty(A^\delta)$ is the set of the completely join-irreducible elements of $A^\delta$, and $R_{\prec}\subseteq \jty(A^\delta)\times \jty(A^\delta)$ such that $(j, i)\in  R_{\prec}$ iff $i\leq \Diamond j$.
\end{definition}
\begin{lemma}
For any subordination lattice $\mathbb{S} =(A, \prec)$, the subordination spaces $\mathbb{S}_*$ given according to the two definitions   above are isomorphic.
\end{lemma}
\begin{proof}
 As is well known, in the canonical extension $A^\delta$ of any  distributive lattice $A$, the set $\jty(A^\delta)$ of the completely join-irreducible elements of $A^\delta$  coincides with the set of its completely join-prime elements, which are in dual order-isomorphism with the prime filters of $A$. Specifically, if $P\subseteq A$ is a prime filter, then $j_P: = \bigwedge P\in K(A^\delta)$ is a completely join-prime element of $A^\delta$; conversely, if $j$ is a completely join-prime element of $A^\delta$, then  $P_j: = \{a\in A\mid j\leq a\}$  is a prime filter of $A$. Clearly, $j = \bigwedge P_j = j_{P_j}$ for any $j\in \jty(A^\delta)$; moreover, it is easy to show, by applying compactness, that $P_{j_{P}} = \{a\in A\mid \bigwedge P\leq a\} = P$ for any prime filter $P$ of $A$.
 
 To complete the proof and show that the two relations $R_{\prec}$ can be identified modulo the identifications above, it is enough to show that  ${\prec}[P]\subseteq Q$ iff $\bigwedge Q\leq \bigwedge {\prec}[P]$ for all prime filters $P$ and $Q$ of $A$. Clearly, ${\prec}[P]\subseteq Q$ implies  $\bigwedge Q\leq \bigwedge {\prec}[P]$. Conversely, if $b\in {\prec} [P]$, then  $\bigwedge Q\leq \bigwedge {\prec}[P]\leq b$, hence, by compactness and $Q$ being an up-set, $b\in Q$, as required.
\end{proof}

In \cite{celani2020subordination}, some properties of subordination lattices are dually characterized in terms  properties of their associated subordination spaces, including those listed in the following proposition, which can be obtained as consequences of the dual characterizations in  Proposition \ref{prop:characteriz}, slanted canonicity \cite{de2021slanted}, and correspondence theory for distributive  modal logic \cite{conradie2012algorithmic}.
\begin{proposition}
\emph{(cf.~\cite{celani2020subordination}, Theorem 5.7)}
\label{prop: celani}
For any subordination lattice $\mathbb{S}$,
\begin{enumerate}[label=(\roman*)]
    \item $\mathbb{S}\models {\prec} \subseteq {\leq}\quad$ iff $\quad R_{\prec}$ is reflexive;
    
    \item $\mathbb{S}\models \mathrm{(D)}\quad $ iff $\quad R_{\prec}$ is transitive, i.e.~$R_{\prec}\circ R_{\prec}\subseteq R_{\prec}$; \item $\mathbb{S}\models \mathrm{(T)}\quad $ iff $\quad R_{\prec}$ is dense, i.e.~$R_{\prec}\subseteq R_{\prec}\circ R_{\prec}$;
    \item  $\mathbb{S}\models (a=\bot) \vee (\blacksquare a \neq \bot)$ iff $R_\prec$ is proper. 
\end{enumerate}
\end{proposition}
\begin{proof}
 (i) By Proposition \ref{prop:characteriz} (i), $\mathbb{S}\models {\prec} \subseteq {\leq}$ iff 
 $\mathbb{S}^{*}\models a\leq \Diamond a$; the inequality $a\leq \Diamond a$ is analytic inductive (cf.~\cite[Definition 55]{greco2018unified}), and hence {\em slanted} canonical by \cite[Theorem 4.1]{de2021slanted}. Hence, from $\mathbb{S}^{*}\models a\leq \Diamond a$ it follows that $(\mathbb{S}^{*})^\delta\models a\leq \Diamond a$, where $(\mathbb{S}^{*})^\delta$  is a {\em standard} (perfect) distributive modal algebra. By algorithmic correspondence theory for distributive modal logic (cf.~\cite[Theorems 8.1 and 9.8]{conradie2012algorithmic}), $(\mathbb{S}^{*})^\delta\models a\leq \Diamond a$ iff $(\mathbb{S}^{*})^\delta\models \forall \nomj (\nomj \leq \Diamond \nomj)$ where $\nomj$ ranges in the set $\jty((\mathbb{S}^{*})^\delta)$. By Definition \ref{def: subordination space in adelta}, this is equivalent to $R_{\prec}$ being reflexive.

(ii) By Proposition  \ref{prop:characteriz}  (iii)(b), $\mathbb{S}\models \mathrm{(D)}$ iff  $\mathbb{S}^{*}\models \Diamond \Diamond a\leq \Diamond a$; the inequality $\Diamond \Diamond a\leq \Diamond a$ is analytic inductive (cf.~\cite[Definition 55]{greco2018unified}), and hence {\em slanted} canonical by \cite[Theorem 4.1]{de2021slanted}. Hence, from $\mathbb{S}^{*}\models \Diamond \Diamond a\leq \Diamond a$ it follows that $(\mathbb{S}^{*})^\delta\models \Diamond \Diamond a\leq \Diamond a$, where $(\mathbb{S}^{*})^\delta$ is a {\em standard} (perfect) distributive modal algebra. By algorithmic correspondence theory for distributive modal logic (cf.~\cite[Theorems 8.1 and 9.8]{conradie2012algorithmic}), $(\mathbb{S}^{*})^\delta\models \Diamond \Diamond a\leq \Diamond a$ iff $(\mathbb{S}^{*})^\delta\models \forall \nomj (\Diamond \Diamond \nomj \leq \Diamond \nomj)$ where $\nomj$ ranges in  $\jty((\mathbb{S}^{*})^\delta)$. The following chain of equivalences holds in any (perfect) algebra $A^\delta$:

\begin{center}
    \begin{tabular}{clll}
    $\forall \nomj (\Diamond\Diamond \nomj\leq \Diamond\nomj)$ & iff & $\forall \nomj\forall \nomk (\nomk\leq \Diamond\Diamond \nomj \Rightarrow \nomk\leq \Diamond\nomj)$ & ($\ast\ast$)\\
    & iff & $\forall \nomj\forall \nomk (\exists \nomi (\nomi\leq \Diamond \nomj \ \& \ \nomk\leq \Diamond\nomi) \Rightarrow \nomk\leq \Diamond\nomj)$ & ($\ast$)\\
    & iff & $\forall \nomj\forall \nomk\forall  \nomi ( (\nomi\leq \Diamond \nomj \ \& \ \nomk\leq \Diamond\nomi) \Rightarrow \nomk\leq \Diamond\nomj)$\\
    \end{tabular}
\end{center}
The equivalence marked with ($\ast\ast$) is due to the fact that canonical extensions of distributive lattices are completely join-generated by the completely join-prime elements.
Let us show the equivalence marked with ($\ast$): as is well known,  every  perfect distributive lattice is  join-generated by its completely join-irreducible elements; hence, $\Diamond j = \bigvee\{i\in \jty(A^\delta)\mid i\leq \Diamond j\}$, and since $\Diamond$ is completely join-preserving, $\Diamond \Diamond j = \bigvee\{\Diamond i\mid i\in \jty(A^\delta)\ \&\ i\leq \Diamond j\}$. Hence, since $k\in \jty(A^\delta)$ is completely join-prime, $k\leq \Diamond \Diamond j$ iff $k\leq \Diamond i$ for some $i\in \jty(A^\delta)$ s.t.~$i\leq \Diamond j$.

By Definition \ref{def: subordination space in adelta}, the last line of the chain of equivalences above is equivalent to $R_{\prec}$ being transitive.

 The proof of (iii) is argued in a similar way using item(iii) (a) of Proposition \ref{prop:characteriz}, and noticing that the modal inequality characterizing conditions (T) is analytic inductive.
 
 (iv)
 
 \begin{tabular}{cll}
    & $R_\prec$ is proper& \\
    iff&  for any proper closed down-set $Y$ of $X(A)$, $\exists P \in X(A)$ such that & \\
    &$ P \neq A\ \&\ P \not \in R_\prec^{-1}[Y]$& \\ 
    iff & for any proper closed down-set $Y$ of $X(A)$, and any $Q \in Y$,  &\\
    & $\exists P \in X(A)$ such that  $ P \neq A\ \&\ \prec{[P]} \not\subseteq Q$& (Definition of $R_\prec$)\\
 \end{tabular}

 Suppose there exists $a \neq \bot$ in $A$ such that  $\prec^{-1}{[a]} \subseteq \{\bot\}$. Then for any $P \in X(A)$, $P \neq A$, we have $\prec [P] \subseteq Q$ for some prime filter $Q$ not containing $a$. Let $Y$ be the set of all  prime filters not containing $a$. Then the above condition does not hold for any $P \in X(A)$.  Therefore, $R_\prec$ is not proper. 
 
 Suppose $\{\bot\}  \subseteq \prec^{-1}{[a]}$ and $\{\bot\}  \neq \prec^{-1}{[a]}$  for all $a \neq \bot \in A$. Let $Q$ be any prime filter not equal to $A$. Then there exists $a_Q \in A$ such that $a_Q \neq \bot$ and $a_Q \not\in Q$. Let $P$ be  proper prime filter generated by $\prec^{-1}{[a_Q]}$. Then, we have $\prec [P] \not\subseteq Q$. Therefore, $R_\prec$ is proper.
 
 Therefore, $R_\prec$ is proper iff  for all $a \neq \bot$ in $A$, $ \{\bot\}  \subseteq \prec^{-1}{[a]} $ and $\{\bot\}  \neq \prec^{-1}{[a]}$ iff $(a =\bot) \vee (\blacksquare a \neq \bot)$.
\end{proof} 
Likewise, items (iv) and (vii) of Proposition \ref{prop:characteriz} can be used to extend Celani's  results and provide  relational characterizations, on subordination spaces, of conditions (CT), (S9), (SL1), (SL2), noticing that the modal inequalities corresponding to those conditions are all analytic inductive.

\begin{proposition}
For any subordination lattice $\mathbb{S}$,
\begin{enumerate}[label=(\roman*)]
    \item $\mathbb{S}\models \mathrm{(CT)}\quad $ iff $\ j R_\prec i$ implies $j R_\prec k, k R_\prec i$, for some $k \leq j$;
    \item $\mathbb{S}\models \mathrm{(S9)} \quad$ iff $\ i_3 R_\prec i_1, i_3 R_\prec i_2 \Leftrightarrow (\exists j \leq i_1)\ j R_\prec i_2, i_3 R_\prec j$;
    \item $\mathbb{S}\models \mathrm{(SL1)}\quad $ iff $\ i_4 R_\prec i_1, i_4 R_\prec i_2, i_3 R_\prec i_4 \Rightarrow (\exists j \leq i_1 \land i_2)\ i_3 R_\prec j$;
    \item $\mathbb{S}\models \mathrm{(SL2)}\quad $ iff $\ i_1 R_\prec i_4, i_2 R_\prec i_4, i_4 R_\prec i_3 \Rightarrow (\exists j \leq i_1 \land i_2)\ j R_\prec i_3$.
\end{enumerate}
\end{proposition}
\begin{proof}
 (i) By Proposition \ref{prop:characteriz} (iv)(a), $\mathbb{S}\models \mathrm{(CT)}$ iff the slanted algebras $\mathbb{S}^{*}\models \Diamond a\leq \Diamond (a \wedge \Diamond a)$; the inequality $\Diamond a\leq \Diamond (a \wedge \Diamond a)$ is analytic inductive, and hence canonical. Hence, from $\mathbb{S}^{*}\models \Diamond a\leq \Diamond (a \wedge \Diamond a)$ it follows that $(\mathbb{S}^{*})^\delta\models \Diamond a\leq \Diamond (a \wedge \Diamond a)$, which is a {\em standard} (perfect) distributive modal algebra. By algorithmic correspondence theory for distributive modal logic, $(\mathbb{S}^{*})^\delta\models \Diamond a\leq \Diamond (a \wedge \Diamond a)$ iff $(\mathbb{S}^{*})^\delta\models \forall \nomj ( \Diamond \nomj \leq \Diamond (\nomj \wedge \Diamond \nomj))$ where $\nomj$ ranges in the set $\jty((\mathbb{S}^{*})^\delta)$ of the completely join-irreducible elements of $(\mathbb{S}^{*})^\delta$. Therefore,
    \begin{center}
    \begin{tabular}{cll}
    & $ \Diamond \nomj \leq \Diamond (\nomj \wedge \Diamond \nomj) $ & \\
    iff& $ \forall \nomi (\nomi \leq \Diamond \nomj \Rightarrow \nomi \leq \Diamond (\nomj \land \Diamond \nomj)$  &
    \\ 
    iff&$ \forall \nomi \exists \nomk (\nomi \leq \Diamond \nomj \Rightarrow \nomi \leq \Diamond \nomk, \nomk \leq \nomj, \nomk \leq \Diamond \nomj)$  &\\ 
    \end{tabular}
    \end{center}
By Definition \ref{def: subordination space in adelta}, the last line of the chain of equivalences above is equivalent to (i).

(ii) By Proposition \ref{prop:characteriz} (vii)(a), $\mathbb{S}\models \mathrm{(S9)}$ iff the slanted algebras $\mathbb{S}^{*}\models \blacksquare ( a \vee \blacksquare b) = \blacksquare a \vee \blacksquare b$; both the inequalities $\blacksquare ( a \vee \blacksquare b) \leq \blacksquare a \vee \blacksquare b$ and  $\blacksquare ( a \vee \blacksquare b) \geq \blacksquare a \vee \blacksquare b$ are analytic inductive, and hence canonical. Hence, from $\mathbb{S}^{*}\models \blacksquare ( a \vee \blacksquare b) = \blacksquare a \vee \blacksquare b$ it follows that $(\mathbb{S}^{*})^\delta\models \blacksquare ( a \vee \blacksquare b) = \blacksquare a \vee \blacksquare b$, which is a {\em standard} (perfect) distributive modal algebra. By algorithmic correspondence theory for distributive modal logic, $(\mathbb{S}^{*})^\delta\models \blacksquare ( a \vee \blacksquare b) = \blacksquare a \vee \blacksquare b$ iff $(\mathbb{S}^{*})^\delta\models \forall \cnomm \forall \cnomn (\blacksquare ( \cnomm \vee \blacksquare \cnomn) = \blacksquare \cnomm \vee \blacksquare \cnomn)$ where $\cnomm, \cnomn$ ranges in the set $\mty((\mathbb{S}^{*})^\delta)$ of the completely meet-irreducible elements of $(\mathbb{S}^{*})^\delta$. Therefore,
\begin{center}
    \begin{tabular}{cll}
    & $ \blacksquare ( \cnomm \vee \blacksquare \cnomn) \leq \blacksquare \cnomm \vee \blacksquare \cnomn $ & \\
    & $\forall \cnomo (\blacksquare \cnomm \vee \blacksquare \cnomn \leq \cnomo \Rightarrow \blacksquare (\cnomm \vee \blacksquare \cnomn) \leq \cnomo)$ & \\
    & $\forall \cnomo (\blacksquare \cnomm \leq \cnomo, \blacksquare \cnomn \leq \cnomo \Rightarrow \exists \cnomo^\prime (\cnomm \leq \cnomo^\prime, \blacksquare \cnomn \leq \cnomo^\prime, \blacksquare \cnomo^\prime \leq \cnomo) ),$ & \\
\end{tabular}
\end{center}
and for the other direction:

\begin{center}
    \begin{tabular}{cll}
    & $\blacksquare \cnomm \vee \blacksquare \cnomn \leq \blacksquare ( \cnomm \vee \blacksquare \cnomn) $ & \\
    & $\forall \cnomo ( \exists \cnomo^\prime (\cnomm \leq \cnomo^\prime, \blacksquare \cnomn \leq \cnomo^\prime, \blacksquare \cnomo^\prime \leq \cnomo) \Rightarrow \blacksquare \cnomm \leq \cnomo, \blacksquare \cnomn \leq \cnomo).$ & \\
\end{tabular}
\end{center}

By Definition \ref{def: subordination space in adelta}, the last line of the chain of equivalences above is equivalent to (ii), given the order-reversing isomorphism $\lambda$ between completely meet-irreducible and completely join-irreducible elements in the distributive setting (it always holds that $\blacksquare \cnomm \leq \cnomn$ iff $\lambda (\cnomm) \leq \Diamond \lambda (\cnomn)$).

We conclude with the proof of (iii), since (iv) is similar. By Proposition \ref{prop:characteriz} (vii)(c), $\mathbb{S}\models \mathrm{(SL1)}$ iff the slanted algebras $\mathbb{S}^{*}\models \blacksquare ( a \vee b) \leq \blacksquare (\blacksquare a \vee \blacksquare b)$; the inequality $\blacksquare ( a \vee b) \leq \blacksquare (\blacksquare a \vee \blacksquare b)$ is analytic inductive, and hence canonical. Hence, from $\mathbb{S}^{*}\models \blacksquare ( a \vee b) \leq \blacksquare (\blacksquare a \vee \blacksquare b)$ it follows that $(\mathbb{S}^{*})^\delta\models \blacksquare ( a \vee b) \leq \blacksquare (\blacksquare a \vee \blacksquare b)$, which is a {\em standard} (perfect) distributive modal algebra. By algorithmic correspondence theory for distributive modal logic, $(\mathbb{S}^{*})^\delta\models \blacksquare ( a \vee b) \leq \blacksquare (\blacksquare a \vee \blacksquare b)$ iff $(\mathbb{S}^{*})^\delta\models \forall \cnomm \forall \cnomn (\blacksquare ( \cnomm \vee \cnomn) \leq \blacksquare (\blacksquare \cnomm \vee \blacksquare \cnomn))$ where $\cnomm, \cnomn$ ranges in the set $\mty((\mathbb{S}^{*})^\delta)$ of the completely meet-irreducible elements of $(\mathbb{S}^{*})^\delta$. Therefore,
\begin{center}
    \begin{tabular}{cll}
    & $ \blacksquare ( \cnomm \vee  \cnomn) \leq \blacksquare (\blacksquare \cnomm \vee \blacksquare \cnomn) $ & \\
    & $ \forall \cnomo (\blacksquare (\blacksquare \cnomm \vee \blacksquare \cnomn) \leq \cnomo \Rightarrow  \blacksquare ( \cnomm \vee  \cnomn) \leq \cnomo) $ & \\
    & $ \forall \cnomo \forall \cnomm_1 (\blacksquare \cnomm \vee \blacksquare \cnomn \leq \cnomm_1, \blacksquare \cnomm_1 \leq \cnomo \Rightarrow \exists \cnomn_1 ( \cnomm \vee \cnomn \leq \cnomn_1, \blacksquare \cnomn_1 \leq \cnomo) )$ &\\
    & $ \forall \cnomo \forall \cnomm_1 (\blacksquare \cnomm \leq \cnomm_1, \blacksquare \cnomn \leq \cnomm_1, \blacksquare \cnomm_1 \leq \cnomo \Rightarrow \exists \cnomn_1 ( \cnomm \leq \cnomn_1, \cnomn \leq \cnomn_1, \blacksquare \cnomn_1 \leq \cnomo) ).$ &\\
 \end{tabular}
    \end{center}
    
    By Definition \ref{def: subordination space in adelta}, the last line of the chain of equivalences above is equivalent to (iii).
\end{proof} 

\section{Conclusions}
\label{sec:conclusions}
We have established a novel connection between the research fields of subordination algebras and of input/output logics and normative reasoning. The present paper focuses only on conditional obligations; however, similarly to the duality between box and diamond operators in modal logic,  conditional permission (aka negative permission) has been introduced and  analysed by Makinson and van der Torre as the dual concept of conditional obligation 
\cite{Makinson03}. In future work, we will study conditional permission in the environment of pre-contact algebras \cite{dimov2005topological},  algebraic structures defined dually to  subordination algebras. 

We have presented a bi-modal characterization of input/output logic in the context of selfextensional logics, a class of logics defined in terms of minimal properties, which are satisfied both by classical propositional logic and by the best known nonclassical logics. %
The present  approach 
is different from other modal formulations of input/output logic \cite{Makinson00,Parent2021,strasser2016adaptive}, also on a non-classical propositional base, in that the output operators themselves are semantically characterized as (suitable restrictions of) modal operators, and their properties characterized in terms of modal axioms (inequalities).  

The bi-modal formulation can be given  different interpretations. For example,  If  $\prec$ satisfies (SI), (WO), (UD) and (DD), then $a\prec b$ iff $\Diamond a\leq b$ iff $a\leq \blacksquare b$, which are the fundamental relations in tense logic. 
Temporal readings of conditional norms is addressed by Makinson in 
\cite{Makinson98}.

In this paper, we have formulated input/output operations in terms of maxiconsistent formulas rather than maxiconsistent sets. The  object-level input/output operations can be used in AGM theory to capture norm dynamics \cite{stolpe2010norm,boella2016agm}. A similar approach  was investigated as a qualitative theory of dynamic interactive belief revision in AGM theory and epistemic logic \cite{baltag2006conditional,baltag2008qualitative}. As a future work, it would be interesting to extend our logical framework to  incorporate a norm-revision mechanism (object-level input/output operations) within dynamic-deontic logic.

Legal Informatics has recently received a lot of attention from industry and institutions due to the rise of RegTech and FinTech. The input/output logic is expressive enough to support reasoning about constitutive, regulative and defeasible rules; these notions play an important role in the legal domains \cite{boella2004regulative}. For example, reified input/output logic is a
suitable formalism for expressing legal statements like those in the General Data
Protection Regulation, for more details see the DAPRECO knowledge base \cite{robaldo2019formalizing}. There are several active projects for implementing input/output reasoners \cite{J46,farjami2020discursive,steen2021goal,lellmann2021input}. One of the current challenges is scalability of legal (and I/O) reasoners. Subordination algebras and its correspondence theorems in first-order logic can be used to understand algorithmic correspondence of input/output operations in first-order logic for designing scalable I/O reasoners, for example see \cite{bellomarini2020vadalog},  for legal applications.

Finally, we hope that the bridge established here can be used to improve mathematical models and methods such as topological, algebraic and
duality-theoretic techniques in normative reasoning on one hand, and finding conceptual applications for subordination algebra and related literature on the other hand. For instance, finding normative meaning for  the new rules discussed here such as (DD), (UD), (SL1), and (SL2) would be interesting.

\appendix
\subsection{Proof of Proposition \ref{prop:characteriz}}
\label{appendix:proof:prop: charact}
\begin{proof}
 (iv)(a) From  left to right, 
Let $a\in A$. If ${\prec}[a] = \varnothing$, then $\Diamond a = \bigwedge \varnothing = \top = \Diamond (a\wedge \top) =\Diamond (a\wedge \Diamond a)$, as required.  
If ${\prec}[a] \neq \varnothing$, then $a\wedge \Diamond a = \bigwedge \{a\wedge e\mid a\prec e\}\in K(A^\delta)$, since, by 
 (DD), $\{a\wedge e\mid a\prec e\}$ is down-directed. Hence,  
 \begin{center}
     \begin{tabular}{rcll}
        $\Diamond (a\wedge \Diamond a)$ 
        &$=$&$\bigwedge\{\Diamond c\mid a\wedge \Diamond a\leq c\}$& definition of $\Diamond$ on $K(A^\delta)$\\
        &$=$&$\bigwedge\{d\mid\exists c(c\prec d\ \& \  a\wedge \Diamond a\leq c)\}$& definition of $\Diamond$ on $A$
     \end{tabular}
 \end{center}  
%
By 
compactness,
$\bigwedge \{a \wedge e \mid a \prec e\}  = a \wedge \Diamond a \leq c$ 
 is equivalent to $a\wedge e\leq c$ for some $e\in A$ such that $a\prec e$. 
Thus, 
\[\Diamond (a\wedge \Diamond a) =  \bigwedge\{d\mid\exists c(c\prec d\ \& \  a\wedge \Diamond a\leq c)\} = \bigwedge\{d\mid\exists c(c\prec d\ \& \ \exists e( a \prec e \ \&\ a \wedge e \leq c))\}.  \] 
To finish the proof that $\Diamond a  = \bigwedge \{d\mid a \prec d\}\leq \Diamond (a\wedge \Diamond a)$, it is enough to show that if $d\in A$ is such that $\exists c(c\prec d\ \& \ \exists e( a \prec e \ \&\ a \wedge e \leq c))$ then $a\prec d$.
 Since $a \wedge e \leq c$ and $c\prec d$, by (SI), $a \wedge e \prec d$.
Hence, by (CT), from 
$a \prec e$ and $a \wedge e \prec d$ it follows that $a\prec d$, as required.

Conversely, assume that $\Diamond a\leq \Diamond (a\wedge \Diamond a)$ holds for any $a$. By (WO), (DD), and Lemma \ref{lem: diamond-output equivalence} (ii), (CT) can be equivalently rewritten as follows:
\[
\text{if }\Diamond a \leq b 
\quad \text{and} \quad
\Diamond (a \wedge b) \leq c, \text{ then } \Diamond a\leq c. 
\]
Since $\wedge$ is monotone, $\Diamond a \leq b $ implies that
$
a \wedge \Diamond a \leq a \wedge b,
$
which implies, by the monotonicity of $\Diamond$ (which is implied by (SI), cf.~Lemma \ref{lem: rleation to axioms}(i)(a)),
that $
\Diamond (a \wedge \Diamond a )\leq \Diamond(a \wedge b).
$
Hence, combining the latter inequality  with $
\Diamond a\leq \Diamond (a\wedge \Diamond a)$ and $\Diamond (a \wedge b) \leq c$, by the transitivity of $\leq$, we get
$\Diamond a\leq c $, as required. The proof of (iv)(b) is dual to that of (vii)(c), and is omitted.

(v) Let $a, b, c\in A$. If $a\prec b$ and $b\prec c$, then  by (SI), $a\wedge b\leq b\prec c$ implies $a\wedge b\prec c$, which, by (CT), implies $a\prec c$, as required.

(vi)  By Proposition \ref{prop:background can-ext}(ii) and (iii), both extensions $\neg^\sigma$ and $\neg^\pi$ on $A^\delta$ are involutive. Hence, the following chains of equivalences hold under both interpretations of the negation, and thus, abusing notation, we omit the superscript.
\begin{center}
\begin{tabular}{rll}
& $\neg \Diamond a \leq \blacksquare \neg a$ & \\
iff & $\neg \blacksquare \neg a \leq \Diamond a = \bigwedge \{b\ |\ a \prec b\}$ &  $\neg$ antitone and involutive \\
iff & $\neg \blacksquare \neg a \leq b$ for all $b\in A$ s.t.~$a \prec b$ &  \\
iff & $\neg b \leq \blacksquare \neg a$ for all $b\in A$ s.t.~$a \prec b$ & $\neg$ antitone and involutive \\
iff & $\neg b \prec \neg a$ for all $b\in A$ s.t.~$a \prec b$ & Lemma \ref{lem: diamond-output equivalence} (iii), (SI), (UD) \\
\end{tabular}
\end{center}

\begin{center}
\begin{tabular}{rll}
& $ \Diamond \neg b \leq \neg \blacksquare  b$ & \\
iff & $   \bigvee \{a\ |\ a \prec b\} = \blacksquare b\leq \neg \Diamond \neg  b$ &  $\neg$ antitone and involutive \\
iff & $a\leq \neg \Diamond \neg  b$ for all $a\in A$ s.t.~$a \prec b$ &  \\
iff & $\Diamond \neg b \leq  \neg a$ for all $a\in A$ s.t.~$a \prec b$ & $\neg$ antitone and involutive \\
iff & $\neg b \prec \neg a$ for all $a\in A$ s.t.~$a \prec b$ & Lemma \ref{lem: diamond-output equivalence} (ii), (WO), (DD) \\
\end{tabular}
\end{center}
 Since $\neg$ is involutive, condition (S6) can equivalently be rewritten as $\forall a\forall b(a\prec \neg b\Rightarrow b\prec \neg a)$ and as $\forall a\forall b( \neg a\prec  b\Rightarrow  \neg b\prec a)$. Hence:
 \begin{center}
\begin{tabular}{rll}
& $\neg \blacksquare  a\leq  \Diamond \neg a = \bigwedge\{b\mid \neg a\prec   b\}  $ & \\
iff & $\neg \blacksquare  a \leq b$ for all $b\in A$ s.t.~$ \neg a \prec b$ &  \\
iff & $\neg b \leq \blacksquare  a$ for all $b\in A$ s.t.~$a \prec b$ & $\neg$ antitone and involutive \\
iff & $\neg b \prec  a$ for all $b\in A$ s.t.~$\neg a \prec b$ & Lemma \ref{lem: diamond-output equivalence} (iii), (SI), (UD) \\
\end{tabular}
\end{center}

\begin{center}
\begin{tabular}{rll}
  & $   \bigvee \{a\ |\ a \prec \neg b\} = \blacksquare \neg b \leq \neg \Diamond b$ &  
  \\
iff & $ a \leq \neg \Diamond b$ for all $b\in A$ s.t.~$a \prec\neg  b$ &  \\
iff & $ \Diamond b \leq  \neg a$ for all $b\in A$ s.t.~$a \prec\neg b$ & $\neg$ antitone and involutive \\
iff & $ b \prec \neg a$ for all $b\in A$ s.t.~$a \prec \neg b$ & Lemma \ref{lem: diamond-output equivalence} (ii), (WO), (DD) \\
\end{tabular}
\end{center}

(vii)(a) From left to right, let $a, b\in A$. If ${\prec}^{-1}[b] = \varnothing$, then $\blacksquare b = \bigvee \varnothing = \bot$, hence $\blacksquare (a\vee \blacksquare b) =\blacksquare (a\vee \bot) = \blacksquare a = \blacksquare a\vee \bot = \blacksquare a\vee \blacksquare b$, as required. If  ${\prec}^{-1}[b] \neq \varnothing$, then
by definition, $a\vee \blacksquare b =\bigvee\{a\vee e\mid e\prec b\} \in O(A^\delta)$ since, by 
(UD), $\{a\vee e\mid e\prec b\}$ is up-directed. Hence:
\begin{center}
\begin{tabular}{rcll}
 $\blacksquare(a \vee \blacksquare b)$&$=$& $\bigvee\{ \blacksquare d\mid d \leq a \vee \blacksquare b\}$& definition of $\blacksquare$ on $O(A^\delta)$\\
 &$=$& $\bigvee\{c \mid \exists d(c\prec d \ \&\ d \leq a \vee \blacksquare b)\}$& definition of $\blacksquare$ on $A$\\
 &&\\
 $\blacksquare a \vee \blacksquare b$&$=$& $ \bigvee\{y\ |\ y \prec a\} \vee \bigvee \{z\ |\ z \prec b\}$\\
 &$=$& $ \bigvee\{y\vee z\mid  y \prec a \text{ and } z \prec b\}$.\\
\end{tabular}
\end{center}
Hence, to  show that $\blacksquare (a \vee \blacksquare b) \leq \blacksquare a \vee \blacksquare b$, it is enough to show that if $c\in A$ is s.t.~$\exists d(c\prec d \ \&\ d \leq a \vee \blacksquare b)$, then $c\leq y\vee z$ for some $y, z\in A$ s.t.~$  y \prec a$  and  $z \prec b$. By compactness, $d \leq a \vee \blacksquare b$ implies that $d\leq a\vee e$ for some $e\in A$ s.t.~$e\prec b$. By (WO), $c\prec d\leq a\vee e$ implies $c\prec a\vee e$. Summing up,  $\exists e (e\prec b\ \&\ c\prec a\vee e)$. Hence, by (S9 $\Rightarrow$), $\exists a'\exists b'(a'\prec a\ \&\ b'\prec b\ \&\ c\leq a'\vee b')$, which is the required condition for $y: = a'$ and $z: = b'$.

Conversely, let $x, a, b \in A$ s.t.~$c\prec b$ and $x \prec a \vee c$ for some $c\in A$. By Lemma  \ref{lem: diamond-output equivalence}(i), $x \prec a \vee c$ implies that $x\leq \blacksquare(a \vee c)$, and $c\prec b$ implies that $c \leq \blacksquare b$. Hence, the monotonicity of $\blacksquare$ (which is guaranteed by (WO), cf.~Lemma \ref{lem: rleation to axioms}(ii)(a)) and the assumption imply that the following chain of inequalities holds: $x\leq \blacksquare(a \vee c)\leq \blacksquare(a \vee \blacksquare b)\leq \blacksquare a \vee \blacksquare b = \left (\bigvee{\prec}^{-1}[a]\right )\vee \left (\bigvee{\prec}^{-1}[b]\right ) = \bigvee\{a'\vee b'\mid a'\prec a\text { and }b'\prec b\}$. 

We claim that $U: =\{a'\vee b'\mid a'\prec a\text { and }b'\prec b\}$ is up-directed: indeed,   if $a'_i\vee b'_i\in U$ for $1\leq i\leq 2$, then $a'_i\prec a$ and $b'_i\prec b$ imply by  (UD) that $a'\prec a$ and $b'\prec b$ for some $a', b'\in A$ s.t.~$a'_i\leq a'$ and $b'_i\leq b'$, hence $a'_i\vee b'_i\leq a'\vee b'\in U$.
Hence, by compactness, $x\leq a'\vee b'$ for some $a', b'\in A$ s.t.~$a'\prec a$ and $b'\prec b$, as required.

(vii)(b) From left to right, let $a, b\in A$. By the definitions spelled out in the proof of (vii)(a), to show that $\blacksquare a\vee \blacksquare b\leq \blacksquare(a\vee \blacksquare b)$, it is enough to show that for all $x, y\in A$, if $y\prec a$ and $z\prec b$, then $y\vee z\prec d$ for some $d\in A$ s.t.~$d\leq a\vee e$ for some $e\in A$ s.t.~$e\prec b$. By (S9 $\Leftarrow$), $y\vee z\prec a\vee c$ for some $c\in A$ s.t.~$c\prec b$. Then the statement is verified for $d: = a\vee c$ and $e: = c$.

Conversely, let $a, b, x\in A$ s.t.~$a'\prec a$, $b'\prec b$ and $x\leq a'\vee b'$ for some $a', b'\in A$, and let us show that $x\prec a\vee c$ for some $c\in A$ s.t.~$c\prec b$.
By Lemma \ref{lem: diamond-output equivalence} (i), the assumptions imply that the following chain of inequalities holds: $x\leq a'\vee b'\leq \blacksquare a\vee \blacksquare b\leq \blacksquare (a\vee \blacksquare b) = \bigvee\{e\mid \exists d(e\prec d\ \& \ d\leq a\vee \blacksquare b)\}$, the last identity being  discussed in the proof of (vii)(a). 

We claim that the set $U:=\{e\mid \exists d(e\prec d\ \& \ d\leq a\vee \blacksquare b)\}$ is up-directed: indeed, if  $e_i\prec d_i$ for  some $d_i\in A$ s.t.~$d_i\leq a\vee \blacksquare b$ where $1\leq i\leq 2$, then by (WO), $e_i\prec d_i\leq d_1\vee d_2$ implies $e_i\prec d_1\vee d_2$, hence by (UD), $e\prec d_1\vee d_2$ for some $e\in A$ s.t.~$e_i\leq e$, and finally, $e\in U$, its witness being $d: = d_1\vee d_2\leq a\vee\blacksquare b$.

Hence, by compactness, $x\leq e$ for some $e\in A$ s.t.~$e\prec d$ for some $d\in A$ s.t.~$d\leq a\vee \blacksquare b= \bigvee\{a\vee c\mid c\prec b\}$. Again by compactness (which is applicable because, as discussed in the proof of (vii)(a), $\{a\vee c\mid c\prec b\}$ is up-directed), $d\leq a\vee c$ for some $c\in A$ s.t.~$c\prec b$. Hence, by (WO) and (SI), $x\leq e\prec d\leq a\vee c$ implies $x\prec a\vee c$, as required.

(vii)(c) From left to right, let $a, b, \in A$. By definition, $\blacksquare a\vee\blacksquare b = \bigvee\{x\vee y\mid x\prec a \text{ and }y\prec b\}\in O(A^\delta)$, since $\{x\vee y\mid x\prec a \text{ and }y\prec b\}$ is up-directed, as discussed in the proof of (vii)(a). Then,
\begin{center}
    \begin{tabular}{r cll}
    $\blacksquare( \blacksquare a\vee\blacksquare b)$ &$=$& $\bigvee \{\blacksquare c\mid c\leq \blacksquare a\vee\blacksquare b\}$ & definition of $\blacksquare$ on $O(A^\delta)$\\
    &$=$& $\bigvee \{d\mid \exists c(d\prec c\ \&\ c\leq \blacksquare a\vee\blacksquare b)\}$. &  $\blacksquare c: = \bigvee \{d\mid d\prec c\}$ \\
    \end{tabular}
\end{center}
Hence, to prove that $\bigvee\{d\mid d\prec a\vee b \} = \blacksquare (a \vee b) \leq \blacksquare (\blacksquare a \vee \blacksquare b)$, it is enough to show that if $d\prec a\vee b$, then $d\prec c$ for some $c\in A$ s.t.~$c\leq \blacksquare a\vee\blacksquare b$.  By (SL1), $d\prec a\vee b$ implies that $d\prec a'\vee b'$ for some $a', b'\in A$ s.t.~$a'\prec a$ and $b'\prec b$. Since $a'\leq \blacksquare a$ and $b'\leq \blacksquare b$, The statement is verified for $c: = a'\vee b'$.

Conversely, let $a, b, c\in A$ s.t.~$c\prec a\vee b$, and let us show that $c\prec a'\vee b'$ for some $a', b'\in A$ s.t.~$a'\prec a$ and $b'\prec b$. From $c\prec a\vee b$ it follows that $c\leq\blacksquare( a\vee b)\leq \blacksquare (\blacksquare a \vee \blacksquare b) = \bigvee \{d\mid \exists e(d\prec e\ \&\ e\leq \blacksquare a\vee\blacksquare b)\}$.

We claim that $U: = \{d\mid \exists e(d\prec e\ \&\ e\leq \blacksquare a\vee\blacksquare b)\}$ is up-directed: indeed, if $d_i\prec e_i$ and $e_i\leq \blacksquare a\vee \blacksquare b$, then $e_1\vee e_2\leq \blacksquare a\vee \blacksquare b$, and by (WO), $d_i\prec e_i\leq e_1\vee e_2$ implies that $d_i\prec e_1\vee e_2$, hence, by (UD), $d\prec e_1\vee e_2$ for some $d\in A$ s.t.~$d_i\leq d$, and $d\in U$, its witness being $e:=e_1\vee e_2$.

Hence, by compactness, $c\leq d$ for some $d\in A$ s.t.~$d\prec e$ for some $e\in A$ s.t.~$e\leq \blacksquare a\vee\blacksquare b= \bigvee\{x\vee y\mid x\prec a \text{ and }y\prec b\}$. Since, s discussed above, $\{x\vee y\mid x\prec a \text{ and }y\prec b\}$ is up-directed, by compactness, $e\leq a'\vee b'$ for some $a', b'\in A$ s.t.~$a'\prec a$ and $b'\prec b$. By (SI) and (WO), $c\leq d\prec e\leq a'\vee b'$ implies $c\prec a'\vee b'$, as required.
\end{proof}
\bibliographystyle{abbrv}
\bibliography{ref}

\end{document}